\documentclass[10pt]{amsart}

\usepackage{amssymb}
\usepackage{amsthm}
\usepackage{amsmath}

\usepackage[usenames]{color}
\definecolor{ANDREW}{RGB}{255,127,0}

\usepackage{tikz}
\usetikzlibrary{arrows,calc}

\usepackage[foot]{amsaddr}

\usepackage{hyperref}

\theoremstyle{plain}
\newtheorem{proposition}{Proposition}[section]
\newtheorem{theorem}[proposition]{Theorem}
\newtheorem{lemma}[proposition]{Lemma}
\newtheorem{corollary}[proposition]{Corollary}
\theoremstyle{definition}
\newtheorem{example}[proposition]{Example}
\newtheorem{definition}[proposition]{Definition}
\newtheorem{observation}[proposition]{Observation}
\theoremstyle{remark}
\newtheorem{remark}[proposition]{Remark}
\newtheorem{conjecture}[proposition]{Conjecture}

\DeclareMathOperator{\Aff}{Aff}

\DeclareMathOperator{\Euc}{Euc}

\DeclareMathOperator{\Cc}{\mathcal{C}}

\DeclareMathOperator{\Hc}{\mathcal{H}}

\DeclareMathOperator{\Cb}{\mathbb{C}}
\DeclareMathOperator{\Db}{\mathbb{D}}

\DeclareMathOperator{\Kb}{\mathbb{K}}

\DeclareMathOperator{\Rb}{\mathbb{R}}
\DeclareMathOperator{\Xb}{\mathbb{X}}

\newcommand{\abs}[1]{\left|#1\right|}

\newcommand{\norm}[1]{\left\|#1\right\|}

\newcommand{\wh}[1]{\widehat{#1}}

\begin{document}

\title[Convex domains not biholomorphic to bounded convex domains]{Kobayashi hyperbolic convex domains not biholomorphic to bounded convex domains}
\author{Andrew Zimmer}\address{Department of Mathematics, Louisiana State University, Baton Rouge, LA, USA}
\email{amzimmer@lsu.edu}
\date{\today}
\subjclass[2010]{}

\begin{abstract} We construct families of convex domains that are biholomorphic to bounded domains, but not bounded convex domains. This is accomplished by finding an obstruction related to the Gromov hyperbolicity of the Kobayashi metric. 
\end{abstract}

\maketitle

\section{Introduction}

 In this paper we construct unbounded convex domains which are Kobayashi hyperbolic, but not biholomorphic to bounded convex domains. This provides a partial answer to a question asked by Forn{\ae}ss-K.T. Kim~\cite[Problem 14]{FK2015} and a negative answer to a question asked by Bracci~\cite{B2020} and Pflug-Zwonek~\cite[Section 5]{PZ2018}.

Kobayashi hyperbolic convex domains are always biholomorphic to bounded domains and share many complex analytic/geometric properties with bounded convex domains~\cite{B1980, BS2009}. Thus there are no obvious obstructions that would prevent a given  Kobayashi hyperbolic convex domain from being biholomorphic to a bounded convex domain. 

We will find such obstructions by considering the asymptotic geometry of the Kobayashi distance $K_\Omega$ on convex domains $\Omega$. We will show that for bounded convex domains, the Kobayashi metric is Gromov hyperbolic if and only if $\Db \times \Db$ does not ``asymptotically isometrically embed'' into the domain. Then we will show that this equivalence fails for certain unbounded Kobayashi hyperbolic convex domains and thus these domains cannot be biholomorphic to bounded convex domains.  
 
Our main result is stated in terms of the geometry of the Hilbert metric $H_C$ on the base $C \subset \Rb^d$ of a convex tube domain. 
  
\begin{theorem}\label{thm:main} If $d \geq 2$, $C \subset \Rb^d$ is a bounded convex domain, and $(C,H_C)$ is Gromov hyperbolic, then $\Omega:=C+i\Rb^d \subset \Cb^d$ is not biholomorphic to a bounded convex domain in $\Cb^d$. 
\end{theorem}

Benoist~\cite{B2003} has characterized the bounded convex domains in $\Rb^d$ with Gromov hyperbolic Hilbert metric and from his classification we have the following corollaries. 

\begin{corollary}If $d \geq 2$ and $C \subset \Rb^d$ is a bounded strongly convex domain, then $\Omega:=C+i\Rb^d$ is not biholomorphic to a bounded convex domain in $\Cb^d$. 
\end{corollary}

\begin{corollary}If $d \geq 2$ and $C \subset \Rb^d$ is a bounded convex domain with real analytic boundary, then $\Omega:=C+i\Rb^d$ is not biholomorphic to a bounded convex domain in $\Cb^d$. 
\end{corollary}

To prove Theorem~\ref{thm:main} we introduce the following asymptotic invariant of a Kobayashi hyperbolic complex manifold.

\begin{definition} Suppose $M,N$ are Kobayashi hyperbolic complex manifolds (e.g. bounded domains in $\Cb^d$). We say that $M$ \emph{asymptotically isometrically embeds into $N$} if there exists a sequence of holomorphic maps $f_n : M \rightarrow N$ such that 
\begin{align*}
K_{M}(z,w)=\lim_{n\rightarrow\infty} K_N( f_n(z),f_n(w))
\end{align*}
locally uniformly on $M \times M$.
\end{definition} 

Using results from~\cite{Z2019b} we will show that for bounded convex domains, the Kobayashi metric is Gromov hyperbolic if and only if $\Db \times \Db$ does not asymptotically isometrically embed into $\Omega$. 

\begin{proposition}\label{prop:bd_convex_char} Suppose $\Omega \subset \Cb^d$ is a bounded convex domain. Then the following are equivalent: 
\begin{enumerate}
\item $(\Omega, K_\Omega)$ is Gromov hyperbolic 
\item $\Db \times \Db$ does not asymptotically isometrically embed into $\Omega$.
\end{enumerate}
\end{proposition}

\begin{remark} The implication $(1) \Rightarrow (2)$ holds for general complex manifolds. As we will see below, the reverse implication fails already for unbounded Kobayashi hyperbolic convex domains. 
\end{remark}

As an aside, we mention the following connection to finite type conditions. In the case when $\Omega \subset \Cb^d$ is a bounded convex domain with $\Cc^\infty$ boundary, the Kobayashi metric is Gromov hyperbolic if and only if $\partial\Omega$ has finite type in the sense of D'Angelo~\cite{Z2016}. Hence we have following corollary. 

\begin{corollary}Suppose $\Omega \subset \Cb^d$ is a bounded convex domain with $\Cc^\infty$ boundary. Then the following are equivalent:
\begin{enumerate}
\item $\partial\Omega$ has finite type in the sense of D'Angelo,
\item $\Db \times \Db$ does not asymptotically isometrically embed into $\Omega$.
\end{enumerate}
\end{corollary}

Using a result of Benoist~\cite{B2003} we will also characterize the tube domains which admit asymptotic isometric embeddings of $\Db \times \Db$. 

\begin{proposition}\label{prop:TD_char} Suppose $d \geq 2$, $C \subset \Rb^d$ is a $\Rb$-properly convex domain, and $\Omega : = C + i \Rb^d$. Then the following are equivalent: 
\begin{enumerate} 
\item $(C,H_C)$ is Gromov hyperbolic,
\item $\Db \times \Db$ does not asymptotically isometrically embed into $\Omega$.
\end{enumerate}
\end{proposition}

Finally, we will use the following characterization of tube domains with Gromov hyperbolic Kobayashi metric.

\begin{theorem}\cite[Corollary 1.13]{Z2019b}\label{thm:TD_gromov} Suppose $d \geq 2$, $C \subset \Rb^d$ is a $\Rb$-properly convex domain, and $\Omega : = C + i \Rb^d$. Then the following are equivalent: 
\begin{enumerate} 
\item $(\Omega, K_\Omega)$ is Gromov hyperbolic,
\item $(C,H_C)$ is Gromov hyperbolic and $C$ is unbounded.
\end{enumerate}
\end{theorem}

Using the above results we can prove Theorem~\ref{thm:main}.

\begin{proof}[Proof of Theorem~\ref{thm:main} assuming Propositions~\ref{prop:bd_convex_char} and~\ref{prop:TD_char}] Suppose $d \geq 2$, $C \subset \Rb^d$ is a bounded convex domain, $(C,H_C)$ is Gromov hyperbolic, and $\Omega:=C+i\Rb^d \subset \Cb^d$. Since $C$ is bounded, Theorem~\ref{thm:TD_gromov} implies that $(\Omega, K_\Omega)$ is not Gromov hyperbolic. Since $(C,H_C)$ is Gromov hyperbolic, Proposition~\ref{prop:TD_char} implies that $\Db \times \Db$ does not asymptotically isometrically embed into $\Omega$. So $\Omega$ cannot be biholomorphic to a bounded convex domain by Proposition~\ref{prop:bd_convex_char}.

\end{proof}

\begin{remark} The proof of Theorem~\ref{thm:main} only requires one part of Theorem~\ref{thm:TD_gromov}: if $C \subset \Rb^d$ is a bounded convex domain and $\Omega:=C+i\Rb^d$, then $(\Omega, K_\Omega)$ is not Gromov hyperbolic. The proof of this is fairly easy, one uses standard estimates on the Kobayashi metric to show that for any $c_0 \in C$ the map 
\begin{align*}
y \in (\Rb^d, d_{\Euc}) \rightarrow c_0 + iy \in (\Omega, K_\Omega)
\end{align*}
is a quasi-isometric embedding (with constants depending on $c_0$). Hence $(\Omega, K_\Omega)$ cannot be Gromov hyperbolic. See~\cite[Lemma 19.4]{Z2019b} for details. 

\end{remark}

There are some examples of tube domains with bounded bases that are biholomorphic to bounded convex domains, for instance:

\begin{example}\label{ex:bad} Let $I =[-1,1] \subset \Rb$ and $C=I^d\subset \Rb^d$. Then 
\begin{align*}
\Omega = C+i\Rb^d =(I+i\Rb)^d 
\end{align*}
is biholomorphic to $\Db^d$. 
\end{example}

This leads to the following conjecture, related to a question of Forn{\ae}ss and K.T. Kim~\cite[Problem 14]{FK2015}.

\begin{conjecture} A tube domain with bounded convex base is biholomorphic to a bounded convex domain if and only if it is biholomorphic to $\Db^d$. \end{conjecture}

For tube domains over unbounded bases, the situation seems more mysterious. Pflug and Zwonek~\cite[Example 18]{PZ2018} proved that the tube domain $\Omega = C+i\Rb^2$ over 
\begin{align*}
C= \{ (x_1,x_2) \in \Rb^2 : x_1, x_2 > 0 \text{ and } x_1 x_2 > 1\}
\end{align*}
is biholomorphic to a bounded convex domain. This domain is non-homogeneous which suggests that the problem of characterizing general tube domains biholomorphic to bounded domains is very difficult. However, from Equation (34) in~\cite{PZ2018} and Theorem 3.1 in~\cite{Z2016} the Kobayashi distance on $\Omega$ is not Gromov hyperbolic. Hence it is possible that the following conjecture is true. 

\begin{conjecture} Suppose $\Omega = C+i\Rb^d$ is a tube domain and $(\Omega, K_\Omega)$ is Gromov hyperbolic (i.e. $C$ is unbounded and $(C,H_C)$ is Gromov hyperbolic). Then $\Omega$ is biholomorphic to a bounded convex domain if and only if $\Omega$ is biholomorphic to the unit ball. \end{conjecture}

 \subsection*{Acknowledgements} This material is based upon work supported by the National Science Foundation under grant DMS-1904099. 

\section{Preliminaries}

\subsection{The space of convex domains} In this section we recall some basic properties of the space of convex domains, for more background see ~\cite{GZ2020}.

Let $\Kb$ be either the real or complex numbers. A convex subset $\Omega \subset \Kb^d$ is called \emph{$\Kb$-properly convex} if every $\Kb$-affine map $f: \Kb \rightarrow \Omega$ is constant. 

Let $\Xb_{d}(\Kb)$ be the set of all $\Kb$-properly convex domains in $\Kb^d$ endowed with the local Hausdorff topology. Then let 
\begin{align*}
\Xb_{d,0}(\Kb) =\left\{ (\Omega, z) \in  \Xb_{d}(\Kb) \times \Kb^d : z \in \Omega\right\} \subset \Xb_d(\Kb) \times \Kb^d.
\end{align*}
When the context is clear we will write $\Xb_d$, $\Xb_{d,0}$ instead of $\Xb_d(\Kb)$, $\Xb_{d,0}(\Kb)$.

\begin{definition} Given a subset $A\subset \Xb_d$ we will let $\overline{A}^{\Xb_d}$ denote its closure in $\Xb_d$. \end{definition}

The group $\Aff(\Kb^d)$ of affine automorphisms of $\Kb^d$ acts continuously on both $\Xb_d$ and $\Xb_{d,0}$. Building upon earlier work of Benz\'{e}cri~\cite{B1959} in the $\Kb=\Rb$ case, Frankel proved that the action of $\Aff(\Kb^d)$ on $\Xb_{d,0}$ is compact.

\begin{theorem}[Frankel~\cite{F1991}]\label{thm:frankel} There exists a compact subset $K \subset \Xb_{d,0}$ such that 
\begin{align*}
\Aff(\Kb^d)\cdot K= \Xb_{d,0}.
\end{align*}
\end{theorem}

We will also use the following observation.

\begin{observation}\label{obs:compact_inclusion} If $\Omega_n$ converges to $\Omega$ in $\Xb_d$ and $K \subset \Omega$ is a compact subset, then
\begin{align*}
K \subset\Omega_n
\end{align*}
for $n$ sufficiently large. 
\end{observation}

\begin{proof} See for instance~\cite[Proposition 3.3]{GZ2020}. \end{proof}

\subsection{The Kobayashi metric} In this section we recall some basic properties of the Kobayashi distance on a $\Cb$-properly convex domain, for more background see~\cite{A1989}. 

Given a domain $\Omega\subset \Cb^d$ let $K_\Omega$ denote the Kobayashi (pseudo-)distance on $\Omega$. For general unbounded domains determining whether or not $K_\Omega$ is non-degenerate is very difficult, but in the special case of convex domains we have the following result of Barth. 

\begin{theorem}[Barth~\cite{B1980}]\label{thm:barth}
Suppose $\Omega$ is a convex domain. Then the following are equivalent:
\begin{enumerate}
\item $\Omega$ is $\Cb$-properly convex, 
\item $\Omega$ is biholomorphic to a bounded domain, 
\item $K_\Omega$ is a non-degenerate distance on $\Omega$, 
\item $(\Omega, K_\Omega)$ is a proper geodesic metric space. 
\end{enumerate}
\end{theorem}

For convex domains, there is a connection between the Kobayashi distance and flat pieces in the boundary. 

\begin{proposition}\cite[Proposition 3.5]{Z2017}\label{prop:finite_dist}
Suppose $\Omega$ is a $\Cb$-properly convex domain and $x, y \in \partial \Omega$ are distinct. Assume $z_m, w_n \in \Omega$ are sequences such that $z_m \rightarrow x$ and $w_n \rightarrow y$. If 
\begin{align*}
\liminf_{m,n \rightarrow \infty} K_{\Omega}(z_m, w_n) < \infty
\end{align*}
and $L$ is the complex line containing $x$ and $y$, then $L \cap \Omega = \emptyset$ and the interior of $\partial\Omega \cap L$ in $L$ contains $x$ and $y$.
\end{proposition}

The Kobayashi distance is also well behaved with respect to the local Hausdorff topology. 

\begin{theorem}\label{thm:dist_convergence} If $\Omega_n$ converges to $\Omega$ in $\Xb_d$, then
\begin{align*}
K_\Omega =\lim_{n \rightarrow \infty} K_{\Omega_n}
\end{align*}
uniformly on compact subsets of $\Omega \times \Omega$. 
\end{theorem}

\begin{proof} See for instance~\cite[Theorem 4.1]{Z2016}. \end{proof}

\subsection{Gromov hyperbolic metric spaces}\label{sec:prelim_gromov} In this section we recall the definition of Gromov hyperbolic metric spaces, for more background see for instance~\cite{BH1999}. 

Suppose $(X,d)$ is a metric space. A curve $\sigma: [a,b] \rightarrow X$ is a \emph{geodesic} if $d(\sigma(t_1),\sigma(t_2)) = \abs{t_1-t_2}$ for all $t_1, t_2 \in [a,b]$.  A \emph{geodesic triangle} in a metric space is a choice of three points in $X$ and geodesic segments  connecting these points. A geodesic triangle is said to be \emph{$\delta$-thin} if any point on any of the sides of the triangle is within distance $\delta$ of the other two sides. 

\begin{definition}
A proper geodesic metric space $(X,d)$ is called \emph{$\delta$-hyperbolic} if every geodesic triangle is $\delta$-thin. If $(X,d)$ is $\delta$-hyperbolic for some $\delta\geq0$ then $(X,d)$ is called \emph{Gromov hyperbolic}.
\end{definition}

In this paper we will also use an equivalent formulation of Gromov hyperbolicity. Given $o,y,z \in X$ the \emph{Gromov product} is 
\begin{align*}
(x|y)_o = \frac{1}{2}(d(o,x)+d(o,y)-d(x,y)).
\end{align*}
Using the Gromov product it is possible to give an alternative definition of Gromov hyperbolicity (for a proof see for instance~\cite[Chapter III.H.1, Proposition 1.22]{BH1999}).

\begin{theorem}
A proper geodesic metric space $(X,d)$ is Gromov hyperbolic if and only if there exists $\alpha \geq 0$ such that 
\begin{align*}
(x|y)_o \geq \min \{ (x|z)_o, (z|y)_o\} - \alpha
\end{align*}
for all $o,x,y,z \in X$.
\end{theorem}

\section{Asymptotic embeddings and analytic disks in the boundary} 

We say that a domain $\Omega \subset \Cb^d$ has \emph{simple boundary} if every holomorphic map $\varphi : \Db \rightarrow \partial \Omega$ is constant. In this section we establish the following characterization of convex domains which admit an asymptotically isometrically embedding of $\Db \times \Db$. 

\begin{theorem}\label{thm:asym_embedding} Suppose $\Omega \subset \Cb^d$ is a $\Cb$-properly convex domain. Then the following are equivalent:
\begin{enumerate}
\item there exists $D\in \overline{\Aff(\Cb^d) \cdot \Omega}^{\Xb_d}$ with non-simple boundary,
\item $\Db \times \Db$ asymptotically isometrically embeds into $\Omega$.
\end{enumerate}
\end{theorem}

The rest of the section is devoted to the proof of Theorem~\ref{thm:asym_embedding}.

\subsection{(1) implies (2)} Suppose $\Omega \subset \Cb^d$ is a $\Cb$-properly convex domain and there exists 
\begin{align*}
D\in \overline{\Aff(\Cb^d) \cdot \Omega}^{\Xb_d}
\end{align*}
with non-simple boundary. 

Let $\Hc= \{ z \in \Cb : {\rm Im}(z) > 0\}$. Then by~\cite[Theorem 7.4]{GZ2020} there exists 
\begin{align*}
D_2\in \overline{\Aff(\Cb^d) \cdot D}^{\Xb_d} \subset \overline{\Aff(\Cb^d) \cdot \Omega}^{\Xb_d}
\end{align*}
with 
\begin{align*}
D_2\cap \left(\Cb^2 \times\{(0,\dots,0)\} \right) = \Hc^2 \times \{(0,\dots,0)\}.
\end{align*}

We will construct an isometric embedding of $\Db \times \Db$ into $D_2$ and use it to show that $\Db \times \Db$ asymptotically isometrically embeds into $\Omega$.

\begin{lemma} If $z,w \in \Hc^2$, then
\begin{align*}
K_{D_2}\Big((z,0,\dots,0),(w,0,\dots,0)\Big) = K_{\Hc^2}(z,w).
\end{align*}
\end{lemma}

\begin{proof} By the distance non-increasing property of the Kobayashi metric, we clearly have
\begin{align*}
K_{D_2}\Big((z,0,\dots,0),(w,0,\dots,0)\Big) \leq K_{\Hc^2}(z,w).
\end{align*}
for all $z,w \in \Hc^2$. 

Let $e_1,\dots,e_d$ be the standard basis of $\Cb^d$. By hypothesis, $(\Rb e_1 + \Cb e_2)\cap D_2 =\emptyset$. Then since $D_2$ is convex, there exists a real hyperplane $H$ such that $H \cap D_2 = \emptyset$ and $(\Rb e_1 + \Cb e_2) \subset H$. Then let $\rho: \Cb^d \rightarrow \Cb$ be the complex linear function with $\rho(z,0,\dots,0)=z$ and $\rho^{-1}(\Rb) = H$. Then $\rho(D_2)=\Hc$ and 
\begin{align*}
\rho^{-1}(0)=H \cap iH \supset \Cb e_2.
\end{align*}
Hence 
\begin{align*}
K_{D_2}\Big( (z_1,z_2,0,\dots,0),(w_1,w_2,0,\dots,0) \Big) & \geq K_{\Hc} \Big(\rho(z_1,z_2,0,\dots,0),\rho(w_1,w_2,0,\dots,0) \Big) \\
& = K_{\Hc}(z_1,w_1)
\end{align*}
for all $(z_1,z_2),(w_1,w_2) \in \Hc^2$. 

Applying the same argument to the second variable shows that 
\begin{align*}
K_{D_2}\Big( (z_1,z_2,0,\dots,0),(w_1,w_2,0,\dots,0) \Big)  \geq K_{\Hc}(z_2,w_2)
\end{align*}
for all $(z_1,z_2),(w_1,w_2) \in \Hc^2$. 

Hence
\begin{align*}
K_{D_2}\Big( (z_1,z_2,0,\dots,0),& (w_1,w_2,0,\dots,0) \Big) \geq \max\left\{ K_{\Hc}(z_1,w_1),K_{\Hc}(z_2,w_2)\right\} \\
& = K_{\Hc^2} \Big((z_1,z_2),(w_1,w_2)\Big)
\end{align*}
for all $(z_1,z_2),(w_1,w_2) \in \Hc^2$. 
\end{proof}

By the previous lemma there exists a holomorphic map $f : \Db \times \Db \rightarrow D_2$ with 
\begin{align*}
K_{\Db \times \Db}(z,w) = K_{D_2}(f(z),f(w))
\end{align*}
for all $z,w \in \Db \times \Db$. Since 
\begin{align*}
D_2\in  \overline{\Aff(\Cb^d) \cdot \Omega}^{\Xb_d},
\end{align*}
there exist affine automorphisms $A_n \in \Aff(\Cb^d)$ with $A_n \Omega \rightarrow D_2$. Using Observation~\ref{obs:compact_inclusion} and passing to a subsequence we can suppose that 
\begin{align*}
f( r_n \cdot \Db \times \Db) \subset A_n \Omega
\end{align*}
where $r_n = 1-1/n$. Next define
\begin{align*}
f_n &: \Db \times \Db \rightarrow \Omega \\
f_n&(z) = A_n^{-1}f(r_nz).
\end{align*}
Then by Theorem~\ref{thm:dist_convergence}
\begin{align*}
K_{\Db \times \Db}(z,w) 
&= K_{D_2}(f(z),f(w))=\lim_{n \rightarrow\infty}K_{D_2}(f(r_nz),f(r_nw))\\
& = \lim_{n \rightarrow\infty}K_{A_n\Omega}(f(r_nz),f(r_nw))\\
& = \lim_{n \rightarrow\infty}K_\Omega(f_n(z),f_n(w))
\end{align*}
locally uniformly on $\Db \times \Db$. Hence $\Db \times \Db$ asymptotically isometrically embeds into $\Omega$.

\subsection{(2) implies (1)} We will use the following lemma. 

\begin{lemma}\label{lem:Fatou} Suppose $\Omega$ is a $\Cb$-properly convex domain and $\varphi : \Db \rightarrow \Omega$ is holomorphic. Then there exists a measurable function $\wh{\varphi} : \partial \Db \rightarrow \overline{\Omega}$ such that: 
\begin{align*}
\wh{\varphi}(e^{i\theta}) = \lim_{r \nearrow 1} \varphi(re^{i\theta})
\end{align*}
for almost every $e^{i\theta} \in \partial \Db$. Moreover, if $\varphi_1,\varphi_2 : \Db \rightarrow \Omega$ are holomorphic and $\wh{\varphi}_1 = \wh{\varphi}_2$ almost everywhere, then $\varphi_1 = \varphi_2$. 
\end{lemma}

Lemma~\ref{lem:Fatou} is a simple consequence of Fatou's Theorem and the Luzin-Privalov radial uniqueness theorem. Delaying the proof of the Lemma until Section~\ref{sec:pf_of_fatou} below, we prove that (2) implies (1) in Theorem~\ref{thm:asym_embedding}. 

Suppose $\Omega \subset \Cb^d$ is a $\Cb$-properly convex domain and there exists a sequence $f_n : \Db \times \Db \rightarrow \Omega$ of holomorphic maps such that 
\begin{align*}
K_{\Db \times \Db}(z,w)=\lim_{n\rightarrow\infty} K_\Omega( f_n(z),f_n(w))
\end{align*}
locally uniformly on $\Db \times \Db$. 

Using Theorem~\ref{thm:frankel} and passing to a subsequence, we can select affine maps $A_n \in \Aff(\Cb^d)$ such that $A_n(\Omega, f_n(0))$ converges to some $(D,z_0)$ in $\Xb_{d,0}$. Then using Theorem~\ref{thm:dist_convergence} and the Arzel\`a-Ascoli theorem, we can pass to a subsequence so that $A_n f_n: \Db \times \Db \rightarrow A_n\Omega$ converges locally uniformly to a holomorphic map $f: \Db \times \Db \rightarrow D$. Then by Theorem~\ref{thm:dist_convergence}
\begin{align*}
K_{\Db \times \Db}(z,w)
&=\lim_{n\rightarrow\infty} K_\Omega( f_n(z),f_n(w))=\lim_{n\rightarrow\infty} K_{A_n\Omega}( A_nf_n(z),A_nf_n(w))\\
& = K_D(f(z),f(w))
\end{align*}
for all $z,w \in \Db \times \Db$.
In particular, $f$ is injective. 

Now fix $w_1,w_2 \in \Db$ distinct and consider the functions $\varphi_1,\varphi_2 : \Db \rightarrow D$ defined by $\varphi_j(\cdot) = f(\cdot,w_j)$. By Lemma~\ref{lem:Fatou} there exist measurable functions $\wh{\varphi}_j : \partial \Db \rightarrow \overline{D}$ and a set $A \subset \partial \Db$ of full measure such that
\begin{align*}
\wh{\varphi}_j(e^{i\theta}) = \lim_{r \nearrow 1} \varphi_j(re^{i\theta})
\end{align*}
for $j=1,2$ and $e^{i\theta} \in A$. Since $f$ is an isometric embedding, we must have $\wh{\varphi}_j(e^{i\theta})  \in \partial D$ when $e^{i\theta} \in A$.

Since $f$ is injective, $\varphi_1 \neq \varphi_2$. So there exists some $e^{i\theta} \in A$ with $\wh{\varphi}_1(e^{i\theta})\neq\wh{\varphi}_2(e^{i\theta})$. However, 
\begin{align*}
 \limsup_{r \nearrow 1} & \ K_{D}(\varphi_1(re^{i\theta}),\varphi_2(re^{i\theta}) ) =  \limsup_{r \nearrow 1} \ K_{\Db \times \Db}\left( (re^{i\theta},w_1) (re^{i\theta},w_2) \right) \\
 &= K_{\Db}(w_1,w_2) < +\infty.
\end{align*}
So by Proposition~\ref{prop:finite_dist}, if $L$ is the complex line containing $\wh{\varphi}_1(e^{i\theta}), \wh{\varphi}_2(e^{i\theta})$, then $L \cap \partial D$ has non-empty interior in $L$. Hence $D$ has non-simple boundary.  

\subsection{Proof of Lemma~\ref{lem:Fatou}}\label{sec:pf_of_fatou}

Suppose $\Omega$ is a $\Cb$-properly convex domain and $\varphi : \Db \rightarrow \Omega$ is holomorphic.

Again let $\Hc= \{ z \in \Cb : {\rm Im}(z) > 0\}$. 

\begin{lemma} There exists an affine map $A \in \Aff(\Cb^d)$ such that $A \Omega \subset \Hc^d$. \end{lemma}

\begin{proof} See for instance~\cite[Proposition 3.5]{BS2009} or~\cite{F1991}. \end{proof}

Consider the maps
\begin{align*}
f : \Hc \rightarrow \Db \\
f(z) = \frac{z-i}{z+i}
\end{align*}
 and $F := (f,\dots, f) \circ A : \Omega \hookrightarrow \Db^d$.  Notice that $F$ extends to a locally bi-Lipschitz homeomorphism of $\overline{\Omega}$ onto its image. Moreover, if 
 \begin{align*}
 Z = \{ (z_1,\dots,z_d) : z_j = 1 \text{ for some } j \},
 \end{align*}
 then 
 \begin{align*}
 \lim_{\norm{z} \rightarrow \infty} d_{\Euc}( F(z),Z) = 0.
 \end{align*}
 
Next consider $\phi:=F \circ \varphi : \Db \rightarrow \Db^d$. Then by Fatou's theorem, there exists a measurable function $\wh{\phi} : \partial \Db \rightarrow \overline{\Db^d}$ such that: 
\begin{align*}
\wh{\phi}(e^{i\theta}) = \lim_{r \nearrow 1} \phi(re^{i\theta})
\end{align*}
for almost every $e^{i\theta} \in \partial \Db$. By modifying $\wh{\phi}$ on a set of measure zero, we can assume that $\wh{\phi}(\partial \Db) \subset F(\overline{\Omega}) \cup Z$. 

\begin{lemma} The set $B=\left\{ e^{i\theta} : \wh{\phi}(e^{i\theta})\in Z\right\}$ has measure zero in $\partial \Db$. \end{lemma}
  
\begin{proof} Suppose for a contradiction that $B$ has positive measure in $\partial \Db$. Let $\pi_j: \Cb^d \rightarrow \Cb$ be the projection onto the $j^{th}$ coordinate. Consider the map 
\begin{align*}
g&: \Db \rightarrow \Cb \\
g&(z) = \prod_{j=1}^d \left( \pi_j(\phi(z)) - 1 \right). 
\end{align*}
Then $g$ is nowhere vanishing since $\phi(\Db) \subset \Db^d$. However 
\begin{align*}
0= \prod_{j=1}^d \left( \pi_j(\wh{\phi}(z)) - 1 \right)= \lim_{r \nearrow 1} g(re^{i\theta})
\end{align*}
for  $e^{i\theta} \in B$. Since $B$ has positive measure and $g$ is bounded, the Luzin-Privalov radial uniqueness theorem implies that $g \equiv 0$. So we have a contradiction. 
\end{proof} 

Then by modifying $\wh{\phi}$ on a set of measure zero, we can assume that $\wh{\phi}(\partial \Db) \subset F(\overline{\Omega})$. Then $\wh{\varphi} := F^{-1} \circ \wh{\phi}$ is well defined and satisfies the lemma.
  
  The ``moreover'' part of the lemma follows from applying the dominated convergence theorem and the Cauchy integral formula to the functions $F \circ \varphi_j$ (or the Luzin-Privalov radial uniqueness theorem). 
  
\section{Proof of Proposition~\ref{prop:bd_convex_char}}

Proposition~\ref{prop:bd_convex_char} is an immediate consequence of Theorem~\ref{thm:asym_embedding} and the following result from~\cite{Z2019b}.

\begin{theorem}\cite[Theorem 1.5]{Z2019b}   Suppose $\Omega \subset \Cb^d$ is a bounded convex domain. Then the following are equivalent:
\begin{enumerate}
\item  $(\Omega, K_\Omega)$ is Gromov hyperbolic,
\item every domain in $\overline{\Aff(\Cb^d) \cdot \Omega}^{\Xb_d}$ has simple boundary.
\end{enumerate}
\end{theorem} 

\section{A result of Benoist}

Recall, a convex domain $C \subset \Rb^d$ is called \emph{strictly convex} if $\partial C$ does not contain any non-trivial line segments. A result of Benoist implies the following characterization of convex domains in $\Rb^d$ with Gromov hyperbolic Hilbert metric. 

\begin{theorem}[{Benoist~\cite[Proposition 1.6]{B2003}}] \label{thm:benoist_thm}Suppose $C \subset \Rb^d$ is a $\Rb$-properly convex domain. Then the following are equivalent:
\begin{enumerate}
\item $(C,H_C)$ is Gromov hyperbolic,
\item every domain in $\overline{\Aff(\Rb^d) \cdot C}^{\Xb_d}$ is strictly convex.
\end{enumerate}
\end{theorem}

To be precise, Theorem~\ref{thm:benoist_thm} is stated differently than~\cite[Proposition 1.6]{B2003}, however the proof is identical. For the reader's convenience we   provide the complete argument. 

\subsection{Preliminaries} We begin by recalling some basic facts about the Hilbert distance. 

Suppose $C \subset \Rb^d$ is a convex domain. Given $x,y \in C$ distinct let $L_{x,y}$ be the real line containing them  and let $a,b \in \partial C \cup \{\infty\}$ be the endpoints of $\overline{C} \cap L_{x,y}$ with the ordering $a,x,y,b$. Then define the \emph{Hilbert pseudo-distance between $x,y$} to be
\begin{align*}
H_C(x,y) = \frac{1}{2} \log \frac{ \norm{x-b}\norm{y-a}}{ \norm{y-b}\norm{x-a}}
\end{align*}
where we define 
\begin{align*}
\frac{ \norm{x-\infty}}{\norm{y-\infty}}=\frac{ \norm{y-\infty}}{\norm{x-\infty}}=1.
\end{align*}
In the case when $C$ does not contain any affine real lines (i.e. $C$ is a $\Rb$-properly convex domain), we see that $H_C(x,y) > 0$ for all $x,y \in C$ distinct. 

We will use the following well known fact about geodesics in the Hilbert metric, for a proof see~\cite[Proposition 2]{dlH1993}.

\begin{proposition} If $C \subset \Rb^d$ is a $\Rb$-properly convex domain and $x,y \in C$, then the Euclidean line segment $[x,y]$ joining $x$ and $y$ can be parametrized to be a unit speed geodesic in $(C,H_C)$. Moreover, if $C$ is strictly convex, then this is the only unit speed geodesic joining $x$ to $y$. 
\end{proposition}

Using the definition of the Hilbert distance is is not difficult to observe that the Hilbert distance is continuous on $\Xb_d$. 

\begin{observation}\label{obs:Hilbert_metric_conv} If $C_n$ converges to $C$ in $\Xb_d$, then
\begin{align*}
H_C =\lim_{n \rightarrow \infty} H_{C_n}
\end{align*}
uniformly on compact subsets of $C \times C$. 
\end{observation}

We will also use the following result of Karlsson-Noskov on the boundary of convex domains with Gromov hyperbolic Hilbert metric. 

\begin{theorem}[Karlsson-Noskov~\cite{KN2002}]\label{thm:KN} Suppose $C \subset \Rb^d$ is a $\Rb$-properly convex domain. If $(C, H_C)$ is Gromov hyperbolic, then 
\begin{enumerate}
\item $C$ is strictly convex,
\item $\partial C$ is a $C^1$ hypersurface. 
\end{enumerate}
\end{theorem}

\subsection{Proof of Theorem~\ref{thm:benoist_thm}}Suppose $C \subset \Rb^d$ is a $\Rb$-properly convex domain.

(1) $\Rightarrow$ (2): Suppose $(C,H_C)$ is Gromov hyperbolic. Then by Observation~\ref{obs:Hilbert_metric_conv} and the Gromov product definition of Gromov hyperbolicity, every domain in $\overline{\Aff(\Rb^d) \cdot C}^{\Xb_d}$ has Gromov hyperbolic Hilbert metric. Hence every domain in the orbit closure is strictly convex by Theorem~\ref{thm:KN}.

(2) $\Rightarrow$ (1): We compactify $\Rb^d$ by adding the sphere at infinity, that is: a sequence $x_n \in \Rb^d$ converges to $\xi \in {\mathbb S}^{d-1}$ if $\norm{x_n} \rightarrow \infty$ and $\frac{1}{\norm{x_n}} x_n \rightarrow \xi$. Then, given $x \in \Rb^d$ and $\xi \in {\mathbb S}^{d-1}$ we can define the intervals
\begin{align*}
(\xi,x) = (x,\xi) = \{ x+ t \xi : t \geq 0\}
\end{align*}
and 
\begin{align*}
(-\xi,\xi)=(\xi,-\xi) = \{ t \xi : t \in \Rb\}.
\end{align*}

By assumption $C$ is strictly convex and hence every geodesic joining two points $x,y \in C$ parametrizes the Euclidean line segment $[x,y]$. Suppose $(C,H_C)$ is not Gromov hyperbolic. Then for each $n \geq 0$ there exist $x_n,y_n,z_n \in C$ and $u_n \in [x_n,y_n]$ such that 
\begin{align*}
H_C(u_n, [x_n,z_n] \cup [z_n,y_n]) \geq n.
\end{align*}
Using Theorem~\ref{thm:frankel}  and passing to a subsequence we can find affine maps $A_n \in \Aff(\Rb^d)$ such that $A_n(C,u_n)$ converges to some $(C_\infty,u_\infty)$ in $\Xb_{d,0}$.  Passing to further subsequences we can assume that $A_nx_n,A_ny_n, A_nz_n$ converges to $x_\infty, y_\infty, z_\infty \in \overline{C}_\infty \cup {\mathbb S}^{d-1}$ respectively. Since 
\begin{align*}
H_C(u_n, \{x_n,y_n, z_n\}) \geq n,
\end{align*}
we must have $x_\infty, y_\infty, z_\infty \in \partial C_\infty \cup {\mathbb S}^{d-1}$.

Since $u_n \in [x_n,z_n] \subset C$ we have $u_\infty \in (x_\infty, y_\infty) \subset C_\infty$. So $x_\infty \neq y_\infty$. Since $C_\infty$ is $\Rb$-properly convex at least one of $x_\infty, y_\infty$ is finite. Then after possibly relabelling we can assume that $x_\infty \neq z_\infty$ and at least one of $x_\infty, z_\infty$ is finite. Then $(x_\infty, z_\infty) \subset \overline{C}_\infty$. Since $C_\infty$ is strictly convex we have $(x_\infty, z_\infty) \subset C_\infty$. 

Now fix $v \in (x_\infty, z_\infty)$. Then there exists $v_n \in (x_n,z_n)$ with $A_n v_n$ converging to $v$. Then 
\begin{align*}
\infty > H_{C_\infty}(v,u_\infty) = \lim_{n \rightarrow \infty} H_{A_n C}(A_nv_n, A_n u_n) = \lim_{n \rightarrow \infty} H_{C}(v_n, u_n) \geq \lim_{n \rightarrow \infty} n = \infty.
\end{align*}
So we have a contradiction and hence $(C,H_C)$ is Gromov hyperbolic. 

\section{Proof of Proposition~\ref{prop:TD_char}}

Proposition~\ref{prop:TD_char} is an immediate consequence of Theorem~\ref{thm:benoist_thm}, Theorem~\ref{thm:asym_embedding}, and the following lemma from~\cite{Z2019b} 

\begin{lemma}\cite[Lemma 19.6]{Z2019b} Suppose $C \subset \Rb^d$ is a $\Rb$-properly convex domain and $\Omega = C + i\Rb^d$. Then 
\begin{align*}
\overline{\Aff(\Cb^d) \cdot \Omega}^{\Xb_d(\Cb)} = \Aff(\Cb^d) \cdot \Big( \overline{\Aff(\Rb^d) \cdot C}^{\Xb_d(\Rb)} + i\Rb^d \Big).
\end{align*}
In particular, the following are equivalent 
\begin{enumerate}
\item every domain in $\overline{\Aff(\Rb^d) \cdot C}^{\Xb_{d}(\Rb)}$ is strictly convex,
\item every domain in $\overline{\Aff(\Cb^d) \cdot \Omega}^{\Xb_{d}(\Cb)}$ has simple boundary. 
\end{enumerate}
\end{lemma}

Notice that the ``in particular'' part is a consequence of the first claim and the following result of Fu-Straube. 

\begin{proposition}[{Fu-Straube~\cite[Theorem 1.1]{FS1998}}]Suppose $\Omega \subset \Cb^d$ is a convex domain. Then every holomorphic map $\Db \rightarrow \partial\Omega$ is constant if and only if every complex affine map $\Db \rightarrow \partial\Omega$ is constant. 
\end{proposition}

\bibliographystyle{alpha}
\bibliography{SCV}

\end{document}